\documentclass[11pt]{amsart}

\usepackage{amssymb}
\usepackage{mathrsfs}
\usepackage{amsfonts}
\usepackage{latexsym,amsmath,amsthm,amssymb,amsfonts}
\usepackage[usenames]{color}
\usepackage{xspace,colortbl}
\usepackage{graphicx}
\usepackage{tipa}

\newcommand{\be}{\begin{equation}}
\newcommand{\ee}{\end{equation}}
\newcommand{\beq}{\begin{eqnarray}}
\newcommand{\eeq}{\end{eqnarray}}

\newtheorem{prop}{Proposition}[section]

\newtheorem{remark}[prop]{Remark}

\def\begeq{\begin{equation}}
\def\endeq{\end{equation}}

\def\odot{\setbox0=\hbox{$\bigcirc$}\relax \mathbin {\hbox
to0pt{\raise.5pt\hbox to\wd0{\hfil $\wedge$\hfil}\hss}\box0 }}

\numberwithin{equation} {section}

\numberwithin{equation}{section}
\newtheorem{theorem}{\bf Theorem}[section]

\newtheorem{lemma}[theorem]{\bf Lemma}

\newtheorem{corollary}[theorem]{\bf Corollary}

\allowdisplaybreaks

\begin{document}

\title[Translating solutions for a class of quasilinear parabolic IBVPs]
 {Translating solutions for a class of quasilinear parabolic initial boundary value problems in Lorentz-Minkowski plane $\mathbb{R}^{2}_{1}$}

\author{
 Ya Gao,~~~Jing-Hua Li,~~~ Jing Mao$^{\ast}$}

\address{
Faculty of Mathematics and Statistics, Key Laboratory of Applied
Mathematics of Hubei Province, Hubei University, Wuhan 430062,
China. }

\email{Echo-gaoya@outlook.com, 1094545647@qq.com, jiner120@163.com}

\thanks{$\ast$ Corresponding author}

\date{}
\begin{abstract}
In this paper, we investigate the evolution of spacelike curves in
Lorentz-Minkowski plane $\mathbb{R}^{2}_{1}$ along prescribed
geometric flows (including the classical curve shortening flow or
mean curvature flow as a special case), which correspond to a class
of quasilinear parabolic initial boundary value problems, and can
prove that this flow exists for all time. Moreover, we can also show
that the evolving spacelike curves converge to a spacelike straight
line or a spacelike Grim Reaper curve as time tends to infinity.
\end{abstract}

\maketitle {\it \small{{\bf Keywords}: Mean curvature flow,
spacelike curves, Lorentz-Minkowski plane, Neumann boundary
condition.}

}

\section{Introduction}

To our knowledge, the start of the study of mean curvature flow (MCF
for short) maybe is due to Brakke \cite{kab} where he used the
geometric measure theory to investigate the motion of surface by its
mean curvature, while Huisken \cite{gh1} (for higher dimensional
case), Gage-Hamilton \cite{ghr} and Grayson \cite{mag} (for lower
dimensional case\footnote{The curve shortening flow (CSF for short)
is the MCF in a prescribed ambient space of dimension $2$, that is,
the CSF is essentially the lower dimensional case of MCF.}) gave
pioneering contributions to this theory, and after that many
interesting related conclusions (or improvements) have been
obtained. BTW, we would like to refer books \cite{cz2,ygi,cm} and
references therein such that readers can have a relatively
comprehensive understanding about the fundamental theory and some
important improvements of the MCF or the CSF.

In order to explain our motivation of writing this paper clearly, we
prefer to give a brief introduction to several results on MCF or CSF
first. They are:

\begin{itemize}

\item (\cite{aw1}) Graphic curves (defined over an interval)
in Euclidean $2$-space $\mathbb{R}^{2}$ satisfying a class of
quasilinear parabolic initial boundary value\footnote{In \cite{aw1},
the boundary value condition of the class of quasilinear parabolic
boundary problems (1.2) considered therein is actually an
inhomogeneous Neumann boundary condition (NBC for short), which at
each endpoint can be described as the tangent value of the contact
angle between the graphic curve and the parabolic boundary. BTW, the
admissible range of the contact angle therein is $(-\pi/2,\pi/2)$,
which implies that the NBC considered in (1.2) of \cite{aw1} is not
necessary to be zero.} problems (IBVPs for short) have been
investigated, and authors therein have proven that the class of
quasilinear parabolic IBVPs considered therein has smooth solution
for $t\in[0,\infty)$ (i.e., has the long-time existence) and its
solution converges as $t\rightarrow\infty$ to a solution moving by
translation with a constant speed. Especially, if two contact angles
are equal, the graphic curve determined by the solution of the class
of quasilinear parabolic IBVPs  converges to a straight line as
$t\rightarrow\infty$. It is easy to check that the standard heat
flow equation and the MCF (or CSF) equation are covered by the class
of quasilinear parabolic IBVPs considered in \cite{aw1} as special
cases. Therefore, as direct consequences, the authors therein showed
that along the heat flow (resp., the MCF), graphic curves (defined
over an interval) in $\mathbb{R}^{2}$ converges as
$t\rightarrow\infty$ to a uniquely determined portion of a parabola
or a straight line (resp., the Grim Reaper or a straight line).

\item (\cite{chmx}) For a given $3$-dimensional Lorentz manifold
$M^{2}\times\mathbb{R}$ with the metric $
\sum_{i,j=1}^{2}\sigma_{ij}dw^{i}\otimes dw^{j}-ds\otimes ds$, where
$M^{2}$ is a $2$-dimensional complete Riemannian manifold with the
metric $\sum_{i,j=1}^{2}\sigma_{ij}dw^{i}\otimes dw^{j}$ and
nonnegative Gaussian curvature, the authors therein investigated the
evolution of spacelike graphs (defined over compact, strictly convex
domains in $M^{2}$) along the nonparametric MCF with prescribed
nonzero NBC, and proved that this flow exists for all the time and
its solutions converge to ones moving only by translation. This
interesting conclusion somehow extends, for instance, the following
results: (Huisken \cite{gh3}) Graphs defined over bounded domains
(with $C^{2,\alpha}$ boundary) in $\mathbb{R}^{n}$ ($n\geq2$), which
are evolving by the MCF with vertical contact angle boundary
condition (i.e., vanishing NBC), have been investigated, and it was
proven that this evolution exists for all the time and the evolving
graphs converge to a constant function as time tends to infinity
(i.e., $t\rightarrow\infty$); (Altschuler-Wu \cite{aw})  Graphs,
defined over strictly convex compact domains in $\mathbb{R}^{2}$,
evolved by the non-parametric MCF with prescribed contact angle (not
necessary to be vertical, i.e., the NBC is not necessary to be
zero), converge to translating surfaces as $t\rightarrow\infty$.

\item Up to  rescalings and isometries of the Lorentz-Minkowski
plane $\mathbb{R}^{2}_{1}$, Halldorsson \cite{hph1} successfully
gave a classification of all self-similar solutions to the MCF of
splacelike curves in $\mathbb{R}^{2}_{1}$, which is a continuation
of his previous work \cite{hph} about the classification of all
self-similar solutions to the CSF of immersed curves in the plane
$\mathbb{R}^{2}$. As explained in \cite[Sect. 1]{hph1}, for the MCF
in $\mathbb{R}^{2}$ and in $\mathbb{R}^{2}_{1}$, there are some
notable differences: (I) in $\mathbb{R}^2$, it is not hard to know
that the length of the evolving simple closed curves is
non-increasing along the MCF, and actually that is the reason why
the MCF in $\mathbb{R}^2$ is called the CSF. However, in
$\mathbb{R}^{2}_{1}$, since the curvature of curves blows up at
lightlike points, the MCF for simple closed curves would not be
considered. One can define the MCF for spacelike (or timelike)
curves with finite Minkowski-length (without having endpoints) in
$\mathbb{R}^{2}_{1}$. For the evolution of spacelike curves (in
$\mathbb{R}^{2}_{1}$) along the MCF, the length of evolving
spacelike curves has the possibility of decreasing or increasing --
see, e.g., \cite[Figure 16]{hph1} for an intuitive explanation. (II)
Halldorsson \cite{hph1} have shown some examples of curves in
$\mathbb{R}^{2}_{1}$ which are initially disjoint but then intersect
under the MCF,  and have also shown examples of non-uniqueness  of
the MCF in $\mathbb{R}^{2}_{1}$. This behavior is totally different
from that of the curves arose in the classification of self-similar
solutions to the MCF in $\mathbb{R}^{2}$, since those curves in
$\mathbb{R}^{2}$ have bounded curvature.

\end{itemize}
Our successful experience on the MCF with nonzero NBC in the Lorentz
$3$-manifold $M^{2}\times\mathbb{R}$ (see \cite{chmx}) and
Halldorsson's works \cite{hph,hph1} motivate us to consider the
evolution of spacelike curves in $\mathbb{R}^{2}_{1}$ and try to
extend Altschuler-Wu's result \cite{aw1} for curves in
$\mathbb{R}^2$.

Denote by $\mathbb{R}^{2}_{1}$ the Lorentz-Minkowski plane with the
Lorentzian metric
\begin{eqnarray*}
\langle\cdot,\cdot\rangle_{L}=dx^{2}-dy^{2}.
\end{eqnarray*}
Let $I_{d}:=[-d,d]$ be a closed interval on the $x$-axis,
$d\in\mathbb{R}^{+}$, and let $\Omega_{d}:=I_{d}\times(0,\infty)$
be the rectangular region in $\mathbb{R}^{2}_{1}$. Clearly,
$\partial I_{d}=\{(-d,0)\}\cup\{(d,0)\}$. For a one-parameter
family of spacelike graphic curves $\mathcal{G}_{t}:=(x,u(x,t))$
defined over $I_{d}$, it is not hard to know that its tangent
vector, the future-directed timelike unit normal vector and the
curvature are given by
\begin{eqnarray} \label{fg}
\vec{e}=(1,u_{x}), \qquad
\vec{\nu}=\frac{(u_{x},1)}{\sqrt{1-u_{x}^{2}}}, \qquad
k=\frac{u_{xx}}{\left(1-u^{2}_{x}\right)^{\frac{3}{2}}},
\end{eqnarray}
where the notations\footnote{This convention implies that
$u_{t}=\frac{\partial u}{\partial t}$,
$u_{xt}=\frac{\partial^{2}u}{\partial x\partial
t}=\frac{\partial^{2}u}{\partial t\partial x}$, and so on. Of
course, here we require that the graphic function $u(x,t)$ has
enough regularity.} $u_{x}=\frac{\partial u}{\partial x}$,
$u_{xx}=\frac{\partial^{2} u}{\partial x^{2}}$ have been used, and
the spacelike assumption implies $|u_{x}|<1$. We investigate the
evolution of $\mathcal{G}_{t}$ along the MCF in
$\mathbb{R}^{2}_{1}$, and can prove the following conclusion.

\begin{theorem}\label{main1.1}
Let $X_{0}:I_{d}\mapsto \mathbb{R}^{2}_{1}$ such that
$\mathcal{G}_{0}:=X_{0}(I_{d})$ can be written as a graph defined
over the interval $I_{d}$. Assume further that
 \begin{eqnarray*}
\mathcal{G}_{0}=\mathrm{graph}_{I_{d}}u_{0}
\end{eqnarray*}
is a spacelike graph over $I_{d}$  for a positive function
$u_{0}:I_{d}\mapsto \mathbb{R}$ satisfying
 \begin{eqnarray} \label{cc-1}
(u_{0})_{x}(-d)=\theta_{-d},~~(u_{0})_{x}(d)=\theta_{d},\qquad
\theta_{i}\in (-1,1),~i=-d,d.
 \end{eqnarray}
 Then the following IBVP
 \begin{equation}\label{1.3}
\left\{
\begin{aligned}
&u_{t}=\frac{u_{xx}}{1-u^{2}_{x}} \qquad && in~I_{d}\times[0,\infty),\\[0.5mm]
&u_{x}(i,t)=\theta_{i} \qquad&& t\in[0,\infty),~i=-d,d\\[0.5mm]
&u(\cdot,0)=u_{0}  \qquad&&  u_{0}\in C^{\infty}(I_{d})
\end{aligned}
\right.
\end{equation}
converges as $t\rightarrow\infty$ to a solution moving by
translation with speed $A(\theta_{-d},\theta_{d},d)$ given by
\begin{eqnarray} \label{speed}
A(\theta_{-d},\theta_{d},d)=\frac{\mathrm{artanh}(\theta_{d})-\mathrm{artanh}(\theta_{-d})}{2d}.
\end{eqnarray}
Moreover, the leaves $\mathcal{G}_{t}$ are spacelike graphs defined
over $I_{d}$, i.e.,
 \begin{eqnarray*}
\mathcal{G}_{t}:=\mathrm{graph}_{I_{d}}u(\cdot, t),
\end{eqnarray*}
and these leaves converge as $t\rightarrow\infty$ to a spacelike
straight line or a spacelike Grim Reaper curve.
\end{theorem}

\begin{remark}
\rm{ (1) The precondition (\ref{cc-1}) is actually the compatibility
condition of the IBVP (\ref{1.3}), which can be used to make sure
the regularity (or smoothness) of the solution $u(\cdot,t)$ to the
IBVP (\ref{1.3}). \\
 (2) In fact, the spacelike assumption for
$\mathcal{G}_{0}$ implies that $(u_{0})_{x}\in(-1,1)$ holds not only
at endpoints $-d, d$ but also on the whole interval $I_{d}$.\\
 (3) For a one-parameter family of spacelike graphic curves
$\mathcal{G}_{t}=(x,u(x,t))$ defined over $I_{d}$ given by the
mapping\footnote{Obviously, this mapping $X$ satisfies
$X(\cdot,0)=X_{0}(\cdot)$.}
$X:I_{d}\times[0,T)\rightarrow\mathbb{R}^{2}_{1}$ for some $T>0$, by
using (\ref{fg}), it is easy to know that the evolution of
$\mathcal{G}_{t}$ along the MCF (with nonzero NBC) in
$\mathbb{R}^{2}_{1}$ can be described by the IBVP (\ref{1.3}). Then
Theorem \ref{main1.1} tells us that this evolution exists for all
the time (i.e., $T=\infty$) and moreover very nice asymptotical
behavior of the evolving curves can be obtained. \\
 (4) To make sure that the RHS of the evolution equation in (\ref{1.3})
does not degenerate, one needs to show $|u_{x}(x,t)|<1$ during the
evolving process (or equivalently, the spacelike property is
preserved under the flow), which will be given by the gradient estimate of Section \ref{S2}.\\
 (5) Clearly, by the property of the inverse hyperbolic tangent function, it is easy
 to
know that $A(\theta_{-d},\theta_{d},d)=0$ if and only if
$\theta_{-d}=\theta_{d}$.\\
 (6) The IBVP (\ref{1.3}) has \emph{smooth} solution on
 $I_{d}\times[0,\infty)$ --  see Subsection
 \ref{subs2.2} for this fact.
 }
\end{remark}

After we have got the main conclusion of Theorem \ref{main1.1}, we
find that by using analysis techniques in Section \ref{S2} here and
Altschuler-Wu's in \cite{aw1}, a more general result can be
obtained. In fact, we have:

\begin{theorem}\label{main3.1}
Let $X_{0}:I_{d}\mapsto \mathbb{R}^{2}_{1}$ such that
$\mathcal{G}_{0}:=X_{0}(I_{d})$ can be written as a graph defined
over the interval $I_{d}$. Assume further that
 \begin{eqnarray*}
\mathcal{G}_{0}=\mathrm{graph}_{I_{d}}u_{0}
\end{eqnarray*}
is a spacelike graph over $I_{d}$  for a positive function
$u_{0}:I_{d}\mapsto \mathbb{R}$ satisfying
 \begin{eqnarray*}
(u_{0})_{x}(-d)=\theta_{-d},~~(u_{0})_{x}(d)=\theta_{d},\qquad
\theta_{i}\in (-1,1),~i=-d,d.
 \end{eqnarray*}
 Then the following IBVP
 \begin{equation}\label{1.1}
\left\{
\begin{aligned}
&u_{t}-(v(u_{x}))_{x}=0 \qquad && in~I_{d}\times[0,\infty),\\[0.5mm]
&u_{x}(i,t)=\theta_{i} \qquad&& t\in[0,\infty),~i=-d,d\\[0.5mm]
&u(\cdot,0)=u_{0}  \qquad&&  u_{0}\in C^{\infty}(I_{d})
\end{aligned}
\right.
\end{equation}
converges as $t\rightarrow\infty$ to a solution moving by
translation with speed $\widetilde{A}(\theta_{-d},\theta_{d},d)$
given by
\begin{eqnarray*}
\widetilde{A}(\theta_{-d},\theta_{d},d)=\frac{v(\theta_{d})-v(\theta_{-d})}{2d},
\end{eqnarray*}
where $v\in C^{\infty}\left((-1,1)\right)$ with its derivative
function satisfying $v'>0$. Moreover, the leaves $\mathcal{G}_{t}$
are spacelike graphs defined over $I_{d}$, i.e.,
 \begin{eqnarray*}
\mathcal{G}_{t}:=\mathrm{graph}_{I_{d}}u(\cdot, t),
\end{eqnarray*}
and these leaves converge as $t\rightarrow\infty$ to a spacelike
straight line or a spacelike Grim Reaper curve.
\end{theorem}

\begin{remark}
\rm{ (1) If $v(u_{x})=\mathrm{artanh}(u_{x})$, then Theorem
\ref{main3.1} degenerates into Theorem \ref{main1.1} directly and
completely. The reason of retaining Theorem \ref{main1.1} here is
that we prefer to show the origin of our idea (on thinking this
topic) to readers clearly. \\
(2) (The heat flow in $\mathbb{R}^{2}_{1}$) For $v(u_{x})=u_{x}$,
the spacelike graph of $u(x,t)$ converges as $t\rightarrow\infty$ to
a uniquely determined portion of a spacelike parabola or a spacelike
straight line. Moreover, in this setting, the translating speed is
$\widetilde{A}(\theta_{-d},\theta_{d},d)=(\theta_{d}-\theta_{-d})/2d$.
}
\end{remark}

\section{The special case: spacelike MCF} \label{S2}

We devote to give the proof of Theorem \ref{main1.1} in this
section.
\subsection{The gradient estimate} \label{Sb2}

Let $0<\alpha<1$ and $T^{\ast}$ be the maximal time such that there
exists some
\begin{eqnarray*}
u\in C^{2+\alpha,1+\frac{\alpha}{2}}(I_{d}\times[0,T^{\ast}))\cap
C^{\infty}(I_{d}\times(0,T^{\ast}))
\end{eqnarray*}
which solves \eqref{1.3}. Next, we shall prove a priori estimates
for those admissible solutions on $[0,T]$, with $T<T^{\ast}$. First,
we have the $C^{0}$ estimate as follows:

\begin{lemma} \label{lemma2-1}
 Let $u$ be a solution of (\ref{1.3}). Then we
have
\begin{eqnarray*}
\inf_{I_{d}}u(x,0)\leq u(x,t)\leq\sup_{I_{d}}u(x,0),\qquad \forall x\in I_{d},~t\in[0,T].
\end{eqnarray*}
\end{lemma}

\begin{proof}
Let $u(x,t)=u(t)$ (independent of $x$) be a solution of (\ref{1.3})
with $u(0)=C$. In this case, the first equation in (\ref{1.3})
reduces to an ordinary differential equation (ODE for short)
\begin{equation*}
\frac{d}{dt}u=0,
\end{equation*}
Therefore
\begin{equation*}
u(t)=C.
\end{equation*}
Using the maximum principle, we can obtain that
\begin{equation*}
\inf_{I_{d}}u(x,0)\leq u(x,t)\leq\sup_{I_{d}}u(x,0).
\end{equation*}
This completes the proof.
\end{proof}

Then we can obtain the $u_{t}$ estimate:

\begin{lemma} \label{lemma2-2}
Let $u$ be a solution of (\ref{1.3}), we have
\begin{eqnarray*}
\inf_{I_{d}}u_{t}(x,0)\leq u_{t}(x,t)\leq\sup_{I_{d}}u_{t}(x,0),\qquad \forall x\in I_{d},~t\in[0,T].
\end{eqnarray*}
\end{lemma}
\begin{proof}
Set
\begin{eqnarray*}
\Phi(x,t)=u_{t}(x,t).
\end{eqnarray*}
Differentiating both sides of the first evolution equation of
(\ref{1.3}), it is easy to get that
 \begin{equation*}
\left\{
\begin{aligned}
&\frac{\partial \Phi}{\partial t}= \frac{1}{1-u^{2}_{x}}\Phi_{xx}+\frac{2u_{x}u_{xx}}{(1-u^{2}_{x})^{2}}\Phi_{x} \qquad && in~I_{d}\times[0,T],\\[0.5mm]
&\Phi_{x}(i,t)=0 \qquad&& t\in[0,T],~i=-d,d\\[0.5mm]
&\Phi(\cdot,0)=\Phi_{0}  \qquad&&  u_{0}\in C^{\infty}(I_{d}).\\
\end{aligned}
\right.
\end{equation*}
Using the maximum principle and Hopf's Lemma, we have
\begin{eqnarray*}
\inf_{I_{d}}u_{t}(x,0)\leq
u_{t}(x,t)\leq\sup_{I_{d}}u_{t}(x,0),\qquad \forall x\in
I_{d},~t\in[0,T],
\end{eqnarray*}
which finishes the proof.
\end{proof}

The gradient estimate can be obtained as follows:

\begin{lemma} \label{lemma2-3}
Let $u$ be a solution of (\ref{1.3}), we have
\begin{eqnarray*}
|u_{x}|(x,t)\leq \sup_{I_{d}}|u_{x}|(x,0)<1\qquad \forall x\in I_{d},~t\in[0,T].
\end{eqnarray*}
\end{lemma}
\begin{proof}
Set $\varphi=\frac{|u_{x}|^2}{2}$. By differentiating  $\varphi$, we
have
\begin{equation*}
\arraycolsep=1.5pt
\begin{array}{lll}
\frac{\partial \varphi}{\partial t}&=&\displaystyle u_{x}u_{xt}\\[3mm]
&=&\displaystyle
u_{x}\left(\frac{u_{xxx}}{1-u^{2}_{x}}+\frac{2u_{x}u^{2}_{xx}}{(1-u_{x}^{2})^{2}}\right)\\[3mm]
&=&\displaystyle
\frac{u_{x}u_{xxx}}{1-u^{2}_{x}}+\frac{2u_{x}^{2}u^{2}_{xx}}{(1-u_{x}^{2})^{2}}.
\end{array}
\end{equation*}
Since
\begin{equation*}
\varphi_{x}=u_{x}u_{xx},~\varphi_{xx}=u^{2}_{xx}+u_{x}u_{xxx},
\end{equation*}
one has $u_{x}u_{xxx}=\varphi_{xx}-u^{2}_{xx}$, and then
\begin{equation*}
\arraycolsep=1.5pt
\begin{array}{lll}
\frac{\partial \varphi}{\partial t}&=&\displaystyle \frac{1}{1-u^{2}_{x}}\varphi_{xx}-\frac{u^{2}_{xx}}{1-u^{2}_{x}}+\frac{2u_{x}^{2}u^{2}_{xx}}{(1-u_{x}^{2})^{2}}\\[3mm]
&=&\displaystyle
\frac{1}{1-u^{2}_{x}}\varphi_{xx}+\frac{2u_{x}u_{xx}}{(1-u^{2}_{x})^{2}}\varphi_{x}-\frac{u^{2}_{xx}}{1-u_{x}^{2}}\\[3mm]
&\leq&\displaystyle
\frac{1}{1-u^{2}_{x}}\varphi_{xx}+\frac{2u_{x}u_{xx}}{(1-u^{2}_{x})^{2}}\varphi_{x}.
\end{array}
\end{equation*}
Then we have
 \begin{equation}\label{2-5}
\left\{
\begin{aligned}
&\frac{\partial \varphi}{\partial t}\leq \frac{1}{1-u^{2}_{x}}\varphi_{xx}+\frac{2u_{x}u_{xx}}{(1-u^{2}_{x})^{2}}\varphi_{x} \qquad && in~I_{d}\times[0,T],\\
&\varphi(i,t)=\frac{\theta_{i}^{2}}{2} \qquad&& t\in[0,T],~i=-d,d\\
&\varphi(\cdot,0)=\frac{|u_{x}(\cdot,0)|^{2}}{2}  \qquad&&  u_{0}\in
C^{\infty}(I_{d}).
\end{aligned}
\right.
\end{equation}
Using the maximum principle to (\ref{2-5}), we have
\begin{equation*}
|\varphi|(x,t)\leq\sup_{I_{d}}|\varphi|(x,0),\qquad\quad \forall x\in I_{d},~t\in[0,T].
\end{equation*}
So,
\begin{equation*}
|u_{x}|(x,t)\leq\sup_{I_{d}}|u_{x}|(x,0),\qquad\quad \forall x\in I_{d},~t\in[0,T].
\end{equation*}
Since $\mathcal{G}_{0}:=\{(x,u(x,0))|x\in I_{d}\}$ is a spacelike
graph of $\mathbb{R}^{2}_{1}$, one has
\begin{equation*}
|u_{x}|(x,t)\leq\sup_{I_{d}}|u_{x}|(\cdot,0)<1,\qquad \forall x\in I_{d},~t\in[0,T].
\end{equation*}
Our proof is finished.
\end{proof}

\begin{corollary}
Let $u$ be a solution of (\ref{1.3}), we have:\\
(1) $0<c_{1}\leq \frac{1}{1-u_{x}^{2}}\leq c_{2}$ holds for some positive constants $c_{1}$, $c_{2}$ depending only on $u_{x}(x,0)$;\\
(2) $u_{xx}$ has uniform bounds depending on $u_{x}(x,0)$ and
$u_{t}(x,0)$.
\end{corollary}
\begin{proof}
This corollary can be easily proven by Lemma \ref{lemma2-2}, Lemma
\ref{lemma2-3} and the evolution equation
$u_{t}=\frac{u_{xx}}{1-u_{x}^{2}}$.
\end{proof}

\subsection{The long-time existence} \label{subs2.2}

 We can get an integral estimate for
$u_{xt}$ as follows:
\begin{lemma} \label{lemma2-5}
For all $\varepsilon\in \mathbb{R}^{+}$, there exists a time $T$
such that for $t\geq T$, the integral $\int u^{2}_{xt}dx(t)$
satisfies $\int u^{2}_{xt}dx(t)\leq\varepsilon$.
\end{lemma}
\begin{proof}
 In fact, by
differentiating $\int u_{t}^{2}dx$, we have
\begin{equation*}
\arraycolsep=1.5pt
\begin{array}{lll}
\frac{d}{dt}\int u^{2}_{t}dx&=&\displaystyle 2\int u_{t}u_{tt}dx\\
&=&\displaystyle
2\int u_{t}\left(\frac{u_{xx}}{1-u_{x}^{2}}\right)_{t}dx\\
&=&\displaystyle -2\int \frac{1}{1-u_{x}^{2}}u^{2}_{xt}dx<0.
\end{array}
\end{equation*}\\
That is, there exists a positive constant $ c_{3}$ (independent of
$t$) such that $\int_0^\infty\int u_{xt}^{2}dxdt\leq c_{3} <\infty$.
It is easy to know that $c_{3}$ depends only on $c_{1}$, $c_{2}$,
and moreover, depends essentially on $u_{x}(x,0)$. Next, we show
that the integral cannot have arbitrarily small spikes in time. By
differentiating $\int u_{xt}^{2}dx$, we have
\begin{equation} \label{2-2-add}
\arraycolsep=1.5pt
\begin{array}{lll}
\frac{d}{dt}\int u^{2}_{xt}dx&=&\displaystyle 2\int u_{xt}u_{ttx}dx\\
&=&\displaystyle
-2\int u_{tt}u_{xxt}dx\\
&=&\displaystyle
-2\int\frac{1}{1-u_{x}^{2}}\left(u_{xxt}^{2}+\frac{2u_{x}u_{xx}u_{xt}u_{xxt}}{1-u_{x}^{2}}\right)dx\\
&\leq&\displaystyle
-2\int\frac{1}{1-u_{x}^{2}}\left[(u_{xxt}+\frac{u_{x}u_{xx}u_{xt}}{1-u_{x}^{2}})^{2}-\frac{u_{x}^{2}u_{xx}^{2}u_{xt}^{2}}{(1-u_{x}^{2})^{2}}\right]dx\\
&\leq&\displaystyle
2\int\frac{1}{1-u_{x}^{2}}u_{x}^{2}u_{xx}^{2}u_{xt}^{2}dx\\
&\leq&\displaystyle c_{4}\int u_{xt}^{2}dx,
\end{array}
\end{equation}
where $c_{4}$ is a positive constant depending only on $u_{x}(x,0)$
and $\sup_{I_{d}}u_{t}(x,0)$. So, from (\ref{2-2-add}), we have
$\int u_{xt}^{2}dx\rightarrow 0$ as $t\rightarrow\infty$.
\end{proof}

By Lemmas \ref{lemma2-2} and \ref{lemma2-5}, we know that
$u_{t}=\frac{u_{xx}}{1-u_{x}^{2}}$ is a uniformly parabolic equation
with H\"{o}lder continuous coefficients. Therefore, by the linear
theory of second-order parabolic PDEs  (see, e.g., \cite[Chap.
4]{Lieb}), there exist some $0<\beta<1$ and some
constant\footnote{For convenience, we will abuse the notation $C$
for constants, and different notations will also be used if
necessary.} $C>0$ such that
$$||u||_{C^{2+\beta,1+\frac{\beta}{2}}(I_{d}\times
[0,T])}\leq C(|| u_0||_{ C^{2+\alpha, 1+\frac{\alpha}{2}}(I_{d})},
\beta, I_{d}).$$ By the Arzel\`{a}-Ascoli theorem, we know that
$u_{T^{\ast}}:=u(\cdot,T^{\ast})$ is also the solution of
(\ref{1.3}). So under the hypothesis of Theorem \ref{main1.1} we
conclude $T^{\ast}=+\infty$. Besides, we can further improve the
regularity of $u$ to $C^{\infty}$ (i.e., from the H\"{o}lder
regularity to the smooth regularity) -- see Lemma \ref{lemma3-5}
with choosing $v(\cdot)=\mathrm{artanh}(\cdot)$ for the proof.

\subsection{The asymptotical behavior}

\begin{lemma} \label{lemma2.6}
\begin{equation}\label{2-9}
-\int|u_{xt}|dx+A(\theta_{-d},\theta_{d},d)\leq u_{t}\leq \int|u_{xt}|dx+A(\theta_{-d},\theta_{d},d).
\end{equation}
\end{lemma}
\begin{proof}
The inequality \eqref{2-9} follows from the estimate
\begin{equation}\label{2-10}
\inf g-\sup g+\frac{\int gdx}{\int dx}\leq g\leq \sup g-\inf g+\frac{\int gdx}{\int dx},
\end{equation}
Replacing $g$ by $u_{t}$ in \eqref{2-10} yields
\begin{equation*}
-\int|u_{xt}|dx+\frac{\int u_{t}dx}{\int dx}\leq u_{t}\leq
\int|u_{xt}|dx+\frac{\int u_{t}dx}{\int dx},
\end{equation*}
Since
$u_{t}=\frac{u_{xx}}{1-u_{x}^{2}}=\left(\mathrm{artanh}(u_{x})\right)_{x}$,
we have
\begin{equation*}
\int
u_{t}dx=\mathrm{artanh}(u_{x}(d,t))-\mathrm{artanh}(u_{x}(-d,t))=\mathrm{artanh}(\theta_{d})-\mathrm{artanh}(\theta_{-d}),
\end{equation*}
and then
\begin{equation*}
-\int|u_{xt}|dx+A(\theta_{-d},\theta_{d},d)\leq u_{t}\leq
\int|u_{xt}|dx+A(\theta_{-d},\theta_{d},d).
\end{equation*}
This completes the proof.
\end{proof}

At the end, we will discuss the asymptotical behavior of the
solution to IBVP (\ref{1.3}) in two cases. In fact, by Lemma
\ref{lemma2-5} and formula (\ref{2-9}), we know that
$u_{t}\rightarrow A(\theta_{-d},\theta_{d},d)$ and
$A(\theta_{-d},\theta_{d},d)\cdot C x^{2}+Bx+D\leq u \leq
A(\theta_{-d},\theta_{d},d)\cdot x^{2}+Bx+D$ as $t\rightarrow\infty$
(i.e., the limiting curve
$u(\cdot,\infty):=\lim_{t\rightarrow}u(\cdot,t)$ should be pinched
by two spacelike parabolas),
where $C$ is a constant depending only on $\sup_{I_{d}}|u_{x}|(\cdot,0)$, and $B$, $D$ are constants. Then: \\
 \textbf{Case 1}. Assume that
$A(\theta_{-d},\theta_{d},d)=0$, i.e., $\theta_{-d}=\theta_{d}$.
Clearly, the graph of $u(x,t)$ converges as $t\rightarrow\infty$ to
a uniquely determined portion of a spacelike straight line. Besides,
if $\theta_{-d}=\theta_{d}=0$, then it is a horizontal line;
$\theta_{-d}=\theta_{d}\neq0$,
it is a straight line with a certain slope.\\
\textbf{Case 2}. Assume that $A(\theta_{-d},\theta_{d},d)\neq 0$,
i.e., $\theta_{-d}\neq\theta_{d}$. Clearly, the graph of $u(x,t)$
converges as $t\rightarrow\infty$ to a uniquely determined portion
of the spacelike Grim Reaper curve.

\section{A general case: a class of quasilinear parabolic initial boundary value problems}

In this section, the proof of Theorem \ref{main3.1} will be shown in
details.

\subsection{The gradient estimate}

Denote by $v=v(u_{x})$ and
$v'=v'(u_{x})=\frac{d}{du_{x}}v(u_{x})$.~From
$u_{t}=(v)_{x}=v'u_{xx}$, we have
\begin{equation*}\label{3-4}
\frac{\partial}{\partial t}u_{x}=v'u_{xxx}+v''u^{2}_{xx},
\end{equation*}
and
\begin{equation}\label{3-5}
\frac{\partial}{\partial t}v=v'v_{xx}.
\end{equation}
Hence, the NBC on $u_{x}$ and the above calculations yield
\begin{equation} \label{3-1-add}
\left\{
\begin{aligned}
&v(u_{x}(i,t))=v(\theta_{i}) \qquad && for~i=-d,d\\
&v_{t}(u_{x}(i,t))=v_{tt}(u_{x}(i,t))=v_{ttt}(u_{x}(i,t))=\ldots=0
\qquad&&  for~i=-d,d.
\end{aligned}
\right.
\end{equation}

As before, let $0<\alpha<1$ and  $T^{\ast}$ be the maximal time such
that there exists some
\begin{eqnarray*}
u\in C^{2+\alpha,1+\frac{\alpha}{2}}(I_{d}\times[0,T^{*}))\cap C^{\infty}(I_{d}\times(0,T^{*}))
\end{eqnarray*}
which solves \eqref{1.1}. Next, we shall prove a priori estimates
for those admissible solutions on $[0,T]$, with $T<T^{\ast}$. First,
we can obtain the following gradient estimate:

\begin{lemma} \label{lemma3-2}
Let $u$ be a solution of (\ref{1.1}), we have
\begin{eqnarray*}
|u_{x}|(x,t)\leq \sup_{I_{d}}|u_{x}|(x,0)<1\qquad \forall x\in I_{d},~t\in[0,T].
\end{eqnarray*}
\end{lemma}
\begin{proof}
Set $\varphi=\frac{|u_{x}|^2}{2}$. By differentiating  $\varphi$, we
have
\begin{equation*}
\arraycolsep=1.5pt
\begin{array}{lll}
\frac{\partial \varphi}{\partial t}&=&\displaystyle u_{x}u_{xt}\\[3mm]
&=&\displaystyle
u_{x}(v'u_{xxx}+v''u_{xx}^{2})\\[3mm]
&=&\displaystyle
v''u_{xx}^{2}u_{x}+v'u_{xxx}u_{x}.
\end{array}
\end{equation*}
Since
\begin{equation*}
\varphi_{x}=u_{x}u_{xx},~\varphi_{xx}=u^{2}_{xx}+u_{x}u_{xxx},
\end{equation*}
one has $u_{x}u_{xxx}=\varphi_{xx}-u^{2}_{xx}$, and then
\begin{equation*}
\arraycolsep=1.5pt
\begin{array}{lll}
\frac{\partial \varphi}{\partial t}&=&\displaystyle v''u_{xx}\varphi_{x}+v'(\varphi_{xx}-u_{xx}^{2})\\[3mm]
&=&\displaystyle
v'\varphi_{xx}+v''u_{xx}\varphi_{x}-v'u_{xx}^{2}\\[3mm]
&\leq&\displaystyle
v'\varphi_{xx}+v''u_{xx}\varphi_{x}.
\end{array}
\end{equation*}
Therefore, we can obtain
 \begin{equation*}
\left\{
\begin{aligned}
&\frac{\partial \varphi}{\partial t}\leq v'\varphi_{xx}+v''u_{xx}\varphi_{x} \qquad && in~I_{d}\times[0,T],\\
&\varphi(i,t)=\frac{\theta_{i}^{2}}{2} \qquad&& t\in[0,T],~i=-d,d\\
&\varphi(\cdot,0)=\frac{|u_{x}(\cdot,0)|^{2}}{2}  \qquad&&  u_{0}\in C^{\infty}(I_{d})\\
\end{aligned}
\right.
\end{equation*}
Using the maximum principle to the above system, one has
\begin{equation*}
|\varphi|(x,t)\leq\sup_{I_{d}}|\varphi|(x,0),\qquad\quad \forall x\in I_{d},~t\in[0,T].
\end{equation*}
So
\begin{equation*}
|u_{x}|(x,t)\leq\sup_{I_{d}}|u_{x}|(x,0),\qquad\quad \forall x\in I_{d},~t\in[0,T].
\end{equation*}
Since $\mathcal{G}_{0}=\{(x,u(x,0))|x\in I_{d}\}$ is a spacelike
graph in $\mathbb{R}^{2}_{1}$, it follows that
\begin{equation*}
|u_{x}|(x,t)\leq\sup_{I_{d}}|u_{x}|(\cdot,0)<1,\qquad \forall x\in
I_{d},~t\in[0,T],
\end{equation*}
which completes the proof.
\end{proof}

Applying the above gradient estimate and the IBVP (\ref{1.1}), it is
not hard to get the following estimates.
\begin{corollary} \label{coro-add}
Let $u$ be a solution of (\ref{1.1}). Then we have:\\
(1) $0<c_{5}\leq v'(t)\leq c_{6}$ holds for some positive $c_{5}$, $c_{6}$ only depending on $u_{x}(x,0)$;\\
(2) $v,v',v'',v'''$ and all higher derivatives have uniform bounds
depending only on $u_{x}(x,0)$.
\end{corollary}

\subsection{The long-time existence}

Combing the system of IBVP (\ref{1.1}) and an almost same argument
to that of \cite[Lemma 2.2]{aw1}, we have:
\begin{lemma} \label{lemma3-4}
For all $\varepsilon\in \mathbb{R}^{+}$, there exists a time $T$
such that for $t\geq T$, the integral $\int v^{2}_{t}dx(t)$
satisfies $\int v^{2}_{t}dx(t)\leq\varepsilon$.
\end{lemma}

Furthermore, the following integral estimates for high-order
derivatives can also be obtained.

\begin{lemma}\label{lemma3-5}
 We have
\begin{equation}\label{3-10}
-\int|v_{xx}|dx+\widetilde{A}(\theta_{-d},\theta_{d},d)\leq
v_{x}\leq \int|v_{xx}|dx+\widetilde{A}(\theta_{-d},\theta_{d},d),
\end{equation}
\begin{equation}\label{3-11}
\sup\left(\frac{\partial^{k}v}{\partial t^{k}}\right)^{2}\leq
2d\int\left|\frac{\partial^{k}v_{x}}{\partial t^{k}}\right|^{2}dx
\qquad for \quad k\geq 1,
\end{equation}
and
\begin{equation}\label{3-12}
\sup\left(\frac{\partial^{k}v_{x}}{\partial t^{k}}\right)^{2}\leq
2d\int\left|\frac{\partial^{k}v_{xt}}{\partial t^{k}}\right|^{2}dx
\qquad for \quad k\geq 1.
\end{equation}
\end{lemma}

\begin{proof}
Similar to Lemma \ref{lemma2.6}, the first inequality \eqref{3-10}
follows from the estimate
\begin{equation*}
\inf g-\sup g+\frac{\int gdx}{\int dx}\leq g\leq \sup g-\inf g+\frac{\int gdx}{\int dx},
\end{equation*}
and the replacement of $v_{x}$ to $g$ in the above estimate. Of
course, in this process the NBC in (\ref{1.1}) has been used.

 Now,
take $h=\frac{\partial^{k}v}{\partial t^{k}}$ and
$h(x_{0})=\max_{I_{d}} h$. In order to show (\ref{3-11}) and
(\ref{3-12}), we need to prove
\begin{equation*}
h^{2}(x_{0})\leq 2d\int h^{2}_{x}dx,
\end{equation*}
that is,
\begin{equation*}
|h(x_{0})|\leq (2d)^{\frac{1}{2}}\left(\int
h^{2}_{x}dx\right)^{\frac{1}{2}}.
\end{equation*}
By (\ref{3-1-add}), one has
\begin{equation*}
\arraycolsep=1.5pt
\begin{array}{lll}
|h(x_{0})|&=&\displaystyle |h(x_{0})-h(-d)|\\[3mm]
&=&\displaystyle
\int_{-d}^{x_{0}}h_{x}dx\\[3mm]
&\leq&\displaystyle
\int_{-d}^d h_{x}dx\\[3mm]
&=&\displaystyle ||h_{x}||_{L^{1}}.
\end{array}
\end{equation*}
Together with the fact
\begin{equation*}
|| h_{x}||_{L^{1}}\leq||1||_{L^{2}}||h_{x}||_{L^{2}},
\end{equation*}
it follows that
\begin{equation*}
|h(x_{0})|\leq ||1||_{L^{2}}||h_{x}||_{L^{2}},
\end{equation*}
which is (\ref{3-11}) exactly. Taking
$s=\frac{\partial^{k}v_{x}}{\partial t^{k}}$ and using a similar
argument, the estimate \eqref{3-12} follows without any difficulty.
\end{proof}

By Lemmas \ref{lemma3-4}, \ref{lemma3-5} and Corollary
\ref{coro-add}, using a similar argument to \cite[Lemma 2.4]{aw1},
we can obtain:

\begin{lemma} \label{lemma3-6}
There exist constants $c_{7}, c_{8}\in \mathbb{R}^{+}$ such that
\begin{equation*}
\int v_{tx}^{2}dx(t)\leq c_{7}e^{-c_{8}t}\int v_{tx}^{2}dx(0),
\end{equation*}
with $c_{7}$, $c_{8}$ depending only on $u_{x}(x,0)$.
\end{lemma}

By Lemmas \ref{lemma3-4}-\ref{lemma3-6}, we know that
$u_{t}=v'u_{xx}$ is a uniformly parabolic equation with H\"{o}lder
continuous coefficients. So under the hypothesis of Theorem
\ref{main3.1}, we conclude $T^{\ast}=+\infty$ and moreover the
solution to the IBVP (\ref{1.1}) is smooth.

\subsection{The asymptotical behavior}
The asymptotical behavior will be discussed in two cases.
 By the fact $\sup (v_{t})^{2}\leq 2d\int |v_{tx}|^{2}dx$ (i.e., the formula (\ref{3-11}) with $k=1$) and Lemma \ref{lemma3-6}, we know that $v_{t}\rightarrow 0$
 exponentially as $t\rightarrow\infty$. Therefore, by formula (\ref{3-10}), $v_{x}\rightarrow\widetilde{A}(\theta_{-d},\theta_{d},d)$
 and $u_{t}=v_{x}\rightarrow\widetilde{A}(\theta_{-d},\theta_{d},d)$ exponentially as $t\rightarrow\infty$, $\widetilde{A}(\theta_{-d},\theta_{d},d)\cdot c_{9} x^{2}+Bx+D\leq u \leq
\widetilde{A}(\theta_{-d},\theta_{d},d)\cdot c_{10} x^{2}+Bx+D$ as
$t\rightarrow\infty$ (i.e., the limiting curve
$u(\cdot,\infty):=\lim_{t\rightarrow}u(\cdot,t)$ should be pinched
by two spacelike parabolas),
where $c_{9}:=\left(\lim_{s\rightarrow1^{-}}v'(s)\right)^{-1}$, $c_{10}:=\left(\lim_{s\rightarrow(-1)^{+}}v'(s)\right)^{-1}$, and  $B$, $D$ are constants. Then:\\
\textbf{Case 1}. Assume that
$\widetilde{A}(\theta_{-d},\theta_{d},d)=0$, i.e.,
$\theta_{-d}=\theta_{d}$. Clearly, the graph of $u(x,t)$ converges
as $t\rightarrow\infty$ to a uniquely determined portion of a
spacelike straight line. Besides, if $\theta_{-d}=\theta_{d}=0$,
then
it is a horizontal line; if $\theta_{-d}=\theta_{d}\neq0$, then it is a straight line with a certain slope.\\
\textbf{Case 2}. Assume that
$\widetilde{A}(\theta_{-d},\theta_{d},d)\neq 0$, i.e.,
$\theta_{-d}\neq\theta_{d}$. Clearly, the graph of $u(x,t)$
converges as $t\rightarrow\infty$ to a uniquely determined portion
of a spacelike Grim Reaper curve.

\section*{Acknowledgments}
This work is partially supported by the NSF of China (Grant Nos.
11801496 and 11926352), the Fok Ying-Tung Education Foundation
(China) and  Hubei Key Laboratory of Applied Mathematics (Hubei
University).


\vspace {1 cm}


\begin{thebibliography}{50}
\setlength{\itemsep}{-0pt} \small

\bibitem{aw1} S. J. Altschuler, L. F. Wu, \emph{Convergence to translating solutions
for a class of quasilinear parabolic boundary problems}, Math. Ann.
{\bf 295} (1993) 761--765.

\bibitem{aw} S. J. Altschuler, L. F. Wu, \emph{Translating surfaces of the non-parametric mean curvature flow with prescribed contact
angle},  Calc. Var. Partial Differential Equations {\bf 2} (1994)
101--111.

\bibitem{kab} K. A. Brakke, \emph{The Motion of a Surface by Its Mean
Curvature}, Math. Notes, vol. 20, Princeton University Press,
Princeton, 1978.

\bibitem{chmx} L. Chen, D. D. Hu, J. Mao, N. Xiang, \emph{Translating surfaces of the non-parametric mean curvature flow in Lorentz manifold
$M^{2}\times\mathbb{R}$}, Chinese Ann. Math., Ser. B {\bf 42}(2)
(2021) 297--310.



\bibitem{cz2} K. S. Chou, X. P. Zhu, \emph{The Curve Shortening
Problem}, Chapman \& Hall/CRC, Boca Raton, 2001.


\bibitem{ghr} M. Gage, R. S. Hamilton, \emph{The heat equation shrinking convex plane
curves}, J. Differential Geom. {\bf23} (1986) 69--96.


\bibitem{ygi} Y. Giga, \emph{Surface Evolution Equations -- A Level Set Approach}, In: Monographs in Mathematics, vol. 99, Birkh\"{a}user
Verlag, 2006.


\bibitem{mag} M. A. Grayson, \emph{The heat equation shrinks embedded plane curves to round
points},  J. Differential Geom. {\bf26} (1987) 285--314.



\bibitem{hph} Hoeskuldur P. Halldorsson, \emph{Self-similar solutions to the curve shortening
flow}, Trans. Amer. Math. Soc. {\bf364}(10) (2012) 5285--5309.


\bibitem{hph1} Hoeskuldur P. Halldorsson, \emph{Self-similar solutions to the mean curvature flow in the Minkowski plane
$\mathbb{R}^{1,1}$}, J. reine angew. Math. {\bf704} (2015) 209--243.


\bibitem{gh1} G. Huisken, \emph{Flow by mean curvature of convex surfaces into
spheres},  J. Differential Geom. {\bf 20} (1984) 237--266.



\bibitem{gh3} G. Huisken, \emph{Non-parametric mean curvature evolution with boundary
conditions},  J. Differential Equat. {\bf 77} (1989) 369--378.


\bibitem{cm} C. Mantegazza, \emph{Lecture Notes on Mean Curvature
Flow}, In: Progress in Mathematics, vol. 290, Birkh\"{a}user Verlag,
2011.


\bibitem{Lieb} G. Lieberman, \emph{Second Order Parabolic Differential
Equations}, World Scientific Publishing Co., 1996.


\end{thebibliography}
\end{document}